\documentclass[10pt,reqno]{amsart}
\usepackage{amssymb}

\usepackage{ifpdf}
\ifpdf
\usepackage[hyperindex]{hyperref}
\else
\expandafter \ifx \csname dvipdfm \endcsname \relax
\usepackage[hypertex,hyperindex]{hyperref}
\else
\usepackage[dvipdfm,hyperindex]{hyperref}
\fi
\fi
\allowdisplaybreaks[4]
\numberwithin{equation}{section}
\newtheorem{thm}{Theorem}[section]

\theoremstyle{definition}
\theoremstyle{remark}
\newtheorem{rem}{Remark}[section]

\begin{document}

\title[A determinantal expression of Delannoy numbers]
{A determinantal expression and a recursive relation of the Delannoy numbers}

\author[F. Qi]{Feng Qi}
\address{Institute of Mathematics, Henan Polytechnic University, Jiaozuo 454010, Henan, China; College of Mathematics, Inner Mongolia University for Nationalities, Tongliao 028043, Inner Mongolia, China; School of Mathematical Sciences, Tianjin Polytechnic University, Tianjin 300387, China}
\email{\href{mailto: F. Qi <qifeng618@gmail.com>}{qifeng618@gmail.com}, \href{mailto: F. Qi <qifeng618@hotmail.com>}{qifeng618@hotmail.com}, \href{mailto: F. Qi <qifeng618@qq.com>}{qifeng618@qq.com}}
\urladdr{\url{https://qifeng618.wordpress.com}}

\dedicatory{Dedicated to people facing and fighting COVID-19}

\begin{abstract}
In the paper, by a general and fundamental, but non-extensively circulated, formula for derivatives of a ratio of two differentiable functions and by a recursive relation of the Hessenberg determinant, the author finds a new determinantal expression and a new recursive relation of the Delannoy numbers. Consequently, the author derives a recursive relation for computing central Delannoy numbers in terms of related Delannoy numbers.
\end{abstract}



\thanks{This paper was typeset using \AmS-\LaTeX}%

\maketitle
\tableofcontents

\section{Motivations}
The Delannoy numbers, denoted by $D(p,q)$ for $p,q\ge0$, form an array of positive integers which are related to lattice paths enumeration and other problems in combinatorics. For more information on their history and status in combinatorics, please refer to~\cite{Cyril-Sylviane-2005-why} and closely related references therein.
\par
In~\cite[Section~2]{Cyril-Sylviane-2005-why} and~\cite{MR3952588}, the explicit formulas
\begin{equation*}
D(p,q)=\sum_{i=0}^{p}\binom{p}{i}\binom{q}{i}2^i
\quad\text{and}\quad
D(p,q)=\sum_{i=0}^{q}\binom{q}{i}\binom{p+q-i}{q}
\end{equation*}
were given.
The first few values of the Delannoy numbers $D(p,q)$ for $0\le p,q\le8$ are listed in Table~\ref{Delannoy-9X9}.
\begin{table}[hbtp]
\caption{Values of the Delannoy numbers $D(p,q)$ for $0\le p,q\le8$} \label{Delannoy-9X9}
\centering
\begin{tabular}{ccccccccc}
 \textbf{1} & 1 & 1 & 1 & 1 & 1 & 1 & 1 & 1 \\
 1 & \textbf{3} & 5 & 7 & 9 & 11 & 13 & 15 & 17 \\
 1 & 5 & \textbf{13} & 25 & 41 & 61 & 85 & 113 & 145 \\
 1 & 7 & 25 & \textbf{63} & 129 & 231 & 377 & 575 & 833 \\
 1 & 9 & 41 & 129 & \textbf{321} & 681 & 1289 & 2241 & 3649 \\
 1 & 11 & 61 & 231 & 681 & \textbf{1683} & 3653 & 7183 & 13073 \\
 1 & 13 & 85 & 377 & 1289 & 3653 & \textbf{8989} & 19825 & 40081 \\
 1 & 15 & 113 & 575 & 2241 & 7183 & 19825 & \textbf{48639} & 108545 \\
 1 & 17 & 145 & 833 & 3649 & 13073 & 40081 & 108545 & \textbf{265729}
\end{tabular}
\end{table}
It is well known~\cite{MR3952588} that the Delannoy numbers $D(p,q)$ satisfy a simple recurrence
\begin{equation}\label{recurs-relat-3-term}
D(p,q)=D(p-1,q)+D(p-1,q-1)+D(p,q-1)
\end{equation}
and can be generated by
\begin{equation*}
\frac1{1-x-y-xy}=\sum_{p,q=0}^\infty D(p,q)x^py^q.
\end{equation*}
When taking $n=p=q$, the numbers $D(n)=D(n,n)$ are known~\cite{MR3952588} as central Delannoy numbers which have the generating function
\begin{equation}\label{G(x)-Delannoy-GF}
\frac1{\sqrt{1-6x+x^2}\,}=\sum_{n=0}^\infty D(n)x^k
=1+3x+13x^2+63x^3+\dotsm.
\end{equation}
The first nine central Delannoy numbers $D(n)$ for $0\le n\le 8$ are listed as blacked numbers in Table~\ref{Delannoy-9X9}.
\par
We found on the \href{https://mathscinet.ams.org/mathscinet/}{MathSciNet} on 27 March 2020 that the phrase ``Delannoy number'' appeared in the titles of the references~\cite{MR3908608, MR2410472, MR2919231, MR3106100, MR3734640, MR3908599, MR2978852, MR2293650, MR2721475, MR3515810, MR3831408, MR3749744, MR2887889, MR2901372, MR2893830, MR2374122, MR1971435, MR2832831, MR3715232, MR2460233, MR1291893, MR3952588}. The Delannoy numbers $D(p,q)$ were also investigated in~\cite{Edwards-Griffiths-2020, MR3168687, Kiselman-15-2, Kiselman-17-2}. The Delannoy numbers $D(p,q)$ have some connections with the Schr\"oder numbers. For some recent results on the Schr\"oder numbers, please refer to~\cite{KJM-706-Turkey.tex, Schroder-Altern.tex, Schroder-Seq-3rd.tex, Integers-2196.tex, IJPA-D-16-00849.tex, MR2598470} and closely related references.
\par
In~\cite[Theorems~1.1 and~1.3]{Delannoy-Cent-P.tex}, considering the generating function~\eqref{G(x)-Delannoy-GF}, among other things, the authors expressed central Delannoy numbers $D(n)$ by an integral
\begin{equation}\label{Delannoy-Integ-Eq}
D(n)=\frac1\pi\int_{3-2\sqrt{2}\,}^{3+2\sqrt{2}\,} \frac1{\sqrt{\bigl(t-3+2\sqrt{2}\,\bigr) \bigl(3+2\sqrt{2}\,-t\bigr)}\,} \frac1{t^{n+1}}\textup{d}t
\end{equation}
and by a determinant
\begin{equation}\label{Delannoy-Det-Eq}
D(n)=(-1)^n
\begin{vmatrix}
a_1 & 1 & 0 & \dotsm &0 &0&0\\
a_2 & a_1 & 1 &\dotsm &0 &0&0\\
a_3 & a_2 & a_1 &\dotsm & 0&0&0\\
\vdots & \vdots & \vdots & \ddots &\vdots&\vdots& \vdots\\
a_{n-2} & a_{n-3} & a_{n-4} &\dotsm &a_1&1& 0\\
a_{n-1} & a_{n-2} & a_{n-3} &\dotsm & a_2&a_1&1\\
a_{n} & a_{n-1} & a_{n-2} &\dotsm &a_3 &a_2& a_1
\end{vmatrix}
\end{equation}
for $n\in\mathbb{N}$, where
\begin{equation*}
a_n=\frac{(-1)^{n+1}}{6^n}\sum_{\ell=1}^n (-1)^{\ell}6^{2\ell} \frac{(2\ell-3)!!}{(2\ell)!!} \binom{\ell}{n-\ell}.
\end{equation*}
Making use of the integral expression~\ref{Delannoy-Integ-Eq}, the authors derived in~\cite{Delannoy-Cent-P.tex} some new analytic properties, including some product inequalities and determinantal inequalities, of central Delannoy numbers $D(n)$.
\par
In combinatorial number theory, it is a significant and meaningful work to express concrete sequences or arrays of integer numbers or polynomials in terms of tridiagonal determinants or the Hessenberg determinants.
In this respect, the Bernoulli numbers and polynomials~\cite{Booth-Nguyen-2008-09, Costabile-Rend-2006, CAM-D-18-00067.tex, 2Closed-Bern-Polyn2.tex, mathematics-131192.tex, Nice-ApBern-Polyn-2v.tex, Turnbull-book-1960}, the Euler numbers and polynomials~\cite{Euler-Polyn-Det-Note.tex, Special-Bell2Euler.tex, Euler-No-3Sum.tex}, (central) Delannoy numbers and polynomials~\cite{Delanoy-No.tex, KJM-706-Turkey.tex, Schroder-Altern.tex, Schroder-Seq-3rd.tex, Integers-2196.tex, IJPA-D-16-00849.tex, UPB-7177.tex, Delannoy-Cent-P.tex}, the Horadam polynomials~\cite{clos-hora-poly-sym.tex}, (generalized) Fibonacci numbers and polynomials~\cite{ipek, Janjic-JIS-2012, Kittappa-LAA-1993, andelic-Fonseca-rev.tex, lematema-1389.tex, UPB-7177.tex, BIMS-D-18-00265.tex}, the Lucas polynomials~\cite{clos-hora-poly-sym.tex}, and the like, have been represented via tridiagonal determinants or the Hessenberg determinants, and consequently many remarkable relations have been obtained. For more information in this area and direction, please refer to~\cite{Hu-Kim-arXiv-2015, Tan-Der-App-Thanks.tex, Fubini-Polyn.tex-v2, AADM-2821.tex, Motzkin-No-Pass.tex, Catalan-GF-Plus.tex, QLG-RACSAM-Ext.tex, KJM-762-Qi-Wang-Guo.tex, 2rdBern-Det-Zhao.tex, Derange-Hess-Det-S.tex} and closely related references therein.
\par
In this paper, by a general and fundamental, but non-extensively circulated, formula for derivatives of a ratio of two differentiable functions in~\cite[p.~40]{Bourbaki-Spain-2004} and by a recursive relation of the Hessenberg determinant in~\cite[p.~222, Theorem]{CollegeMJ-2002-Cahill}, we find a new determinantal expression and a new recursive relation of the Delannoy numbers $D(p,q)$. Consequently, we derive a recursive relation for computing central Delannoy numbers $D(n)$ in terms of related Delannoy numbers $D(p,q)$.

\section{A determinantal expression of the Delannoy numbers}
In this section, by virtue of a general and fundamental, but non-extensively circulated, formula for derivatives of a ratio of two differentiable functions in~\cite[p.~40]{Bourbaki-Spain-2004}, we find a new determinantal expression of the Delannoy numbers $D(p,q)$.

\begin{thm}\label{delannoy-det-exp-thm}
For $p,q\ge0$, the Delannoy numbers $D(p,q)$ can be determinantally expressed by
\begin{equation}\label{Delannoy-det-exp-eq}
D(p,q)=\frac{(-1)^q}{q!}\begin{vmatrix}L_{(q+1)\times1}(p) & M_{(q+1)\times q}(p)\end{vmatrix}_{(q+1)\times(q+1)},
\end{equation}
where
\begin{align*}
L_{(q+1)\times1}(p)&=\begin{pmatrix}\langle p\rangle_0,\langle p\rangle_1,\dotsc,\langle p\rangle_q\end{pmatrix}^T,\\
M_{(q+1)\times q}(p)&=\begin{pmatrix}(-1)^{i-j}\binom{i-1}{j-1}\langle p+1\rangle_{i-j}
\end{pmatrix}_{\substack{1\le i\le q+1\\ 1\le j\le q}},\\
\langle z\rangle_n&=
\begin{cases}
z(z-1)\dotsm(z-n+1), & n\ge1;\\
1,& n=0
\end{cases}
\end{align*}
is known as the $n$-th falling factorial of the number $z\in\mathbb{C}$, and $T$ denotes the transpose of a matrix. Consequently, central Delannoy numbers $D(n)$ for $n\ge0$ can be determinantally expressed as
\begin{equation}\label{central-recurs-eq}
D(n)=\frac{(-1)^n}{n!}\begin{vmatrix}L_{(n+1)\times1}(n) & M_{(n+1)\times n}(n)\end{vmatrix}_{(n+1)\times(n+1)}.
\end{equation}
\end{thm}

\begin{proof}
We recall a general and fundamental, but non-extensively circulated, formula for derivatives of a ratio of two differentiable functions.
Let $u(t)$ and $v(t)\ne0$ be two $n$-th differentiable functions for $n\in\mathbb{N}$. Exercise~5) in~\cite[p.~40]{Bourbaki-Spain-2004} reads that the $n$-th derivative of the ratio $\frac{u(t)}{v(t)}$ can be computed by
\begin{equation}\label{Sitnik-Bourbaki-reform}
\frac{\textup{d}^n}{\textup{d}x^n}\biggl[\frac{u(t)}{v(t)}\biggr]
=(-1)^n\frac{\bigl|W_{(n+1)\times(n+1)}(t)\bigr|}{v^{n+1}(t)},
\end{equation}
where $U_{(n+1)\times1}(t)$ is an $(n+1)\times1$ matrix whose elements satisfy $u_{k,1}(t)=u^{(k-1)}(t)$ for $1\le k\le n+1$, $V_{(n+1)\times n}(t)$ is an $(n+1)\times n$ matrix whose elements meet $v_{i,j}(t)=\binom{i-1}{j-1}v^{(i-j)}(t)$ for $1\le i\le n+1$ and $1\le j\le n$, and $|W_{(n+1)\times(n+1)}(t)|$ is the determinant of the $(n+1)\times(n+1)$ matrix
\begin{equation*}
W_{(n+1)\times(n+1)}(t)=\begin{pmatrix}U_{(n+1)\times1}(t) & V_{(n+1)\times n}(t)\end{pmatrix}_{(n+1)\times(n+1)}.
\end{equation*}
\par
It is easy to see that
\begin{equation*}
\frac{\partial^p}{\partial x^p}\biggl(\frac1{1-x-y-xy}\biggr)
=\frac{p!(1+y)^p}{[1-x-(1+x)y]^{p+1}}.
\end{equation*}
Making use of the formula~\eqref{Sitnik-Bourbaki-reform} gives
\begin{gather*}
\frac{\partial^{p+q}}{\partial y^q\partial x^p}\biggl(\frac1{1-x-y-xy}\biggr)
=p!\frac{\partial^{q}}{\partial y^q}\frac{(1+y)^p}{[1-x-(1+x)y]^{p+1}}\\
=p!\frac{(-1)^q}{[1-x-(1+x)y]^{(p+1)(q+1)}}\\
\times\left|
\begin{matrix}
(1+y)^p & [1-x-(1+x)y]^{p+1}\\
\langle p\rangle_1(1+y)^{p-1} &(-1)^1\langle p+1\rangle_1(1+x)^1[1-x-(1+x)y]^{p}\\
\langle p\rangle_2(1+y)^{p-2} &(-1)^2\langle p+1\rangle_2(1+x)^2[1-x-(1+x)y]^{p-1}\\
\langle p\rangle_3(1+y)^{p-3} &(-1)^3\langle p+1\rangle_3(1+x)^3[1-x-(1+x)y]^{p-2}\\
\vdots &\vdots\\
\langle p\rangle_{q-2}(1+y)^{p-q+2} &(-1)^{q-2}\langle p+1\rangle_{q-2}(1+x)^{q-2}[1-x-(1+x)y]^{p-q+3}\\
\langle p\rangle_{q-1}(1+y)^{p-q+1} &(-1)^{q-1}\langle p+1\rangle_{q-1}(1+x)^{q-1}[1-x-(1+x)y]^{p-q+2}\\
\langle p\rangle_q(1+y)^{p-q} &(-1)^q\langle p+1\rangle_q(1+x)^q[1-x-(1+x)y]^{p-q+1}
\end{matrix}\right.\\
\begin{matrix}
0\\
\binom{1}{1}[1-x-(1+x)y]^{p+1}\\
\binom{2}{1}(-1)^1\langle p+1\rangle_1(1+x)^1[1-x-(1+x)y]^{p}\\
\binom{3}{1}(-1)^2\langle p+1\rangle_2(1+x)^2[1-x-(1+x)y]^{p-1}\\
\vdots\\
\binom{q-2}{1}(-1)^{q-3}\langle p+1\rangle_{q-3}(1+x)^{q-3}[1-x-(1+x)y]^{p-q+4}\\
\binom{q-1}{1}(-1)^{q-2}\langle p+1\rangle_{q-2}(1+x)^{q-2}[1-x-(1+x)y]^{p-q+3}\\
\binom{q}{1}(-1)^{q-1}\langle p+1\rangle_{q-1}(1+x)^{q-1}[1-x-(1+x)y]^{p-q+2}
\end{matrix}\\
\begin{matrix}
0&\dotsm\\
0&\dotsm\\
\binom{2}{2}[1-x-(1+x)y]^{p+1}&\dotsm\\
\binom{3}{2}(-1)^1\langle p+1\rangle_1(1+x)^1[1-x-(1+x)y]^{p}&\dotsm\\
\vdots&\ddots\\
\binom{q-2}{2}(-1)^{q-4}\langle p+1\rangle_{q-4}(1+x)^{q-4}[1-x-(1+x)y]^{p-q+5}&\dotsm\\
\binom{q-1}{2}(-1)^{q-3}\langle p+1\rangle_{q-3}(1+x)^{q-3}[1-x-(1+x)y]^{p-q+4}&\dotsm\\
\binom{q}{2}(-1)^{q-2}\langle p+1\rangle_{q-2}(1+x)^{q-2}[1-x-(1+x)y]^{p-q+3}&\dotsm
\end{matrix}\\
\begin{matrix}
0                                                            \\
0                                                            \\
0                                                            \\
0                                                            \\
\vdots                                                       \\
\binom{q-2}{q-2}[1-x-(1+x)y]^{p+1}                           \\
\binom{q-1}{q-2}(-1)\langle p+1\rangle_1(1+x)[1-x-(1+x)y]^{p}\\
\binom{q}{q-2}\langle p+1\rangle_2(1+x)^2[1-x-(1+x)y]^{p-1}
\end{matrix}\\
\left.
\begin{matrix}
                                                          0\\
                                                          0\\
                                                          0\\
                                                          0\\
                                                     \vdots\\
                             0\\
\binom{q-1}{q-1}[1-x-(1+x)y]^{p+1}\\
\binom{q}{q-1}(-1)\langle p+1\rangle_1(1+x)[1-x-(1+x)y]^{p}
\end{matrix}\right|\\
\to(-1)^qp!
\left|
\begin{matrix}
\langle p\rangle_0     & (-1)^0\langle p+1\rangle_0        & 0                                 \\
\langle p\rangle_1     & (-1)^1\langle p+1\rangle_1        & \binom{1}{1}(-1)^0\langle p+1\rangle_0          \\
\langle p\rangle_2     & (-1)^2\langle p+1\rangle_2        & \binom{2}{1}(-1)^1\langle p+1\rangle_1          \\
\langle p\rangle_3     & (-1)^3\langle p+1\rangle_3        & \binom{3}{1}(-1)^2\langle p+1\rangle_2           \\
\vdots                 & \vdots                            & \vdots           \\
\langle p\rangle_{q-2} & (-1)^{q-2}\langle p+1\rangle_{q-2}& \binom{q-2}{1}(-1)^{q-3}\langle p+1\rangle_{q-3} \\
\langle p\rangle_{q-1} & (-1)^{q-1}\langle p+1\rangle_{q-1}& \binom{q-1}{1}(-1)^{q-2}\langle p+1\rangle_{q-2} \\
\langle p\rangle_q     & (-1)^{q}\langle p+1\rangle_q      & \binom{q}{1}(-1)^{q-1}\langle p+1\rangle_{q-1}
\end{matrix}\right.\\
\left.
\begin{matrix}
0 & \dotsm &                                                     0&                                                          0\\
0& \dotsm &                                                      0&                                                          0\\
\binom{2}{2}(-1)^0\langle p+1\rangle_0& \dotsm &                 0&                                                          0\\
\binom{3}{2}(-1)^1\langle p+1\rangle_1& \dotsm &                 0&                                                          0\\
\vdots& \ddots &                                                 \vdots&                                                     \vdots\\
\binom{q-2}{2}(-1)^{q-4}\langle p+1\rangle_{q-4}& \dotsm &       \binom{q-2}{q-2}(-1)^0\langle p+1\rangle_0&                             0\\
\binom{q-1}{2}(-1)^{q-3}\langle p+1\rangle_{q-3}& \dotsm &       \binom{q-1}{q-2}(-1)^1\langle p+1\rangle_1 &\binom{q-1}{q-1}(-1)^0\langle p+1\rangle_0\\
\binom{q}{2}(-1)^{q-2}\langle p+1\rangle_{q-2} & \dotsm & \binom{q}{q-2}(-1)^2\langle p+1\rangle_2 &\binom{q}{q-1}(-1)^1\langle p+1\rangle_1
\end{matrix}\right|
\end{gather*}
as $x,y\to0$. Consequently, we have
\begin{multline*}
D(p,q)=\frac{1}{p!q!}\frac{\partial^{p+q}}{\partial y^q\partial x^p}\biggl(\frac1{1-x-y-xy}\biggr)\\
=\frac{(-1)^q}{q!}
\begin{vmatrix}\begin{pmatrix}\langle p\rangle_{ij}\end{pmatrix}_{\substack{0\le i\le q\\ j=1}}
&\begin{pmatrix}(-1)^{i-j}\binom{i-1}{j-1}\langle p+1\rangle_{i-j}
\end{pmatrix}_{\substack{1\le i\le q+1\\ 1\le j\le q}}
\end{vmatrix}_{(q+1)\times(q+1)}.
\end{multline*}
The determinantal expression~\eqref{Delannoy-det-exp-eq} is thus proved.
\par
From~\eqref{Delannoy-det-exp-eq}, we readily see that, when $n=p=q$, central Delannoy numbers $D(n)$ for $n\ge0$ can be expressed as~\eqref{central-recurs-eq}.
The proof of Theorem~\ref{delannoy-det-exp-thm} is complete.
\end{proof}

\begin{rem}
Since the symmetric property $D(p,q)=D(q,p)$, from the determinantal expression~\eqref{Delannoy-det-exp-eq} in Theorem~\ref{delannoy-det-exp-thm}, it follows that
\begin{multline*}
\begin{vmatrix}L_{(q+1)\times1}(p) & M_{(q+1)\times q}(p)\end{vmatrix}_{(q+1)\times(q+1)}\\
=(-1)^{p-q}\frac{q!}{p!}\begin{vmatrix}L_{(p+1)\times1}(q) & M_{(p+1)\times p}(q)\end{vmatrix}_{(p+1)\times(p+1)}.
\end{multline*}
For example, when $p=8$ and $q=3$, we have
\begin{equation*}
\begin{vmatrix}L_{4\times1}(8) & M_{4\times3}(8)\end{vmatrix}_{4\times4}
=-\frac{3!}{8!}\begin{vmatrix}L_{9\times1}(3) & M_{9\times8}(3)\end{vmatrix}_{9\times9}.
\end{equation*}
This means that, when computing $D(p,q)$, if $p\ge q$, it would be sufficient to compute the determinant
\begin{equation*}
\begin{vmatrix}L_{(q+1)\times1}(p) & M_{(q+1)\times q}(p)\end{vmatrix}_{(q+1)\times(q+1)}.
\end{equation*}
\end{rem}

\begin{rem}
The determinantal expression~\eqref{central-recurs-eq} is different from and simpler than~\eqref{Delannoy-Det-Eq} established in~\cite[Theorem~1.1]{Delannoy-Cent-P.tex}.
\end{rem}

\section{A recursive relation of the Delannoy numbers}
In this section, by virtue of a recursive relation of the Hessenberg determinant in~\cite[p.~222, Theorem]{CollegeMJ-2002-Cahill}, we will find a new recursive relation of the Delannoy numbers $D(p,q)$.

\begin{thm}\label{Delannoy-recursive-thm}
For $p,q\ge0$, the Delannoy numbers $D(p,q)$ satisfy the recursive relation
\begin{equation}\label{Delannoy-recursive-eq}
D(p,q)=\binom{p}{q} +(-1)^{q-1}\sum_{r=0}^{q-1}(-1)^r\binom{p+1}{q-r}D(p,r).
\end{equation}
Consequently, central Delannoy numbers $D(n)$ for $n\ge0$ satisfy
\begin{equation}\label{Delan-central-delan}
D(n)=1+(-1)^{n+1}\sum_{r=0}^{n-1}(-1)^r\binom{n+1}{r+1}D(n,r).
\end{equation}
\end{thm}

\begin{proof}
Let $Q_0=1$ and
\begin{equation*}
Q_n=
\begin{vmatrix}
e_{1,1} & e_{1,2} & 0 & \dotsc & 0 & 0\\
e_{2,1} & e_{2,2} & e_{2,3} & \dotsc & 0 & 0\\
e_{3,1} & e_{3,2} & e_{3,3} & \dotsc & 0 & 0\\
\vdots & \vdots & \vdots & \vdots & \vdots & \vdots\\
e_{n-2,1} & e_{n-2,2} & e_{n-2,3} & \dotsc & e_{n-2,n-1} & 0 \\
e_{n-1,1} & e_{n-1,2} & e_{n-1,3} & \dotsc & e_{n-1,n-1} & e_{n-1,n}\\
e_{n,1} & e_{n,2} & e_{n,3} & \dotsc & e_{n,n-1} & e_{n,n}
\end{vmatrix}
\end{equation*}
for $n\in\mathbb{N}$. In~\cite[p.~222, Theorem]{CollegeMJ-2002-Cahill}, it was proved that the sequence $Q_n$ for $n\ge0$ satisfies $Q_1=e_{1,1}$ and
\begin{equation}\label{CollegeMJ-2002-Cahill-Thm}
Q_n=\sum_{r=1}^n(-1)^{n-r}e_{n,r} \Biggl(\prod_{j=r}^{n-1}e_{j,j+1}\Biggr) Q_{r-1}
\end{equation}
for $n\ge2$, where the empty product is understood to be $1$. Replacing the determinant $Q_r$ by $(-1)^{r-1}(r-1)!D(p,r-1)$ in~\eqref{Delannoy-det-exp-eq} for $1\le r\le n$ in the recursive relation~\eqref{CollegeMJ-2002-Cahill-Thm} and simplifying give
\begin{equation*}
D(p,n-1)=\frac{\langle p\rangle_{n-1}}{(n-1)!}
+(-1)^n\sum_{r=2}^n(-1)^r\frac{\langle p+1\rangle_{n-r+1}}{(n-r+1)!}D(p,r-2)
\end{equation*}
which is equivalent to the recursive relation~\eqref{Delannoy-recursive-eq}.
\par
When $n=p=q$ in~\eqref{Delannoy-recursive-eq}, we can see that central Delannoy numbers $D(n)$ satisfy the recursive relation~\eqref{Delan-central-delan}.
The proof of Theorem~\ref{Delannoy-recursive-thm} is complete.
\end{proof}

\begin{rem}
The recursive relation~\eqref{Delannoy-recursive-eq} is different from~\eqref{recurs-relat-3-term}.
\end{rem}

\begin{rem}
The recursive relation~\eqref{Delan-central-delan} demonstrates that we can compute central Delannoy numbers $D(n)$ in terms of the Delannoy numbers $D(p,q)$. However, the Delannoy numbers $D(p,q)$ can not be expressed in terms of central Delannoy numbers $D(n)$.
\end{rem}

\begin{rem}
The recursive relation~\eqref{Delan-central-delan} can be rearranged as
\begin{equation*}
D(n)=1-\sum_{r=0}^{n-1}(-1)^{n-r}\binom{n+1}{r+1}D(n,r)
=1-\sum_{r=1}^n(-1)^r\binom{n+1}{r}D(n,n-r).
\end{equation*}
\end{rem}

\begin{rem}
The recursive relation~\eqref{Delan-central-delan} can also be written as
\begin{align*}
1-D(1)&=\frac{-1}{1!}\langle2\rangle_1D(1,0),\\
1-D(2)&=\frac{-1}{1!}\langle3\rangle_1D(2,1)+\frac{1}{2!}\langle3\rangle_2D(2,0),\\
1-D(3)&=\frac{-1}{1!}\langle4\rangle_1D(3,2)+\frac{1}{2!}\langle4\rangle_2D(3,1) +\frac{-1}{3!}\langle4\rangle_3D(3,0),\\
1-D(4)&=\frac{-1}{1!}\langle5\rangle_1D(4,3)+\frac{1}{2!}\langle5\rangle_2D(4,2) +\frac{-1}{3!}\langle5\rangle_3D(4,1) +\frac{1}{4!}\langle5\rangle_4D(4,0),
\end{align*}
and so on. Consequently, making use of the symmetry relation $D(p,q)=D(q,p)$, we can reformulate the recursive relation~\eqref{Delan-central-delan} as
\begin{equation*}
\begin{pmatrix}
D(1)\\
D(2)\\
D(3)\\
D(4)\\
\vdots\\
D(n-1)\\
D(n)
\end{pmatrix}
=
\begin{pmatrix}
1\\
1\\
1\\
1\\
\vdots\\
1\\
1
\end{pmatrix}
-
M_{n\times n}
\begin{pmatrix}
\frac{-1}{1!}\\
\frac{1}{2!}\\
\frac{-1}{3!}\\
\frac{1}{4!}\\
\vdots\\
\frac{(-1)^{n-1}}{(n-1)!}\\
\frac{(-1)^n}{n!}
\end{pmatrix},
\end{equation*}
where
\begin{gather*}
M_{n\times n}=\left(
\begin{matrix}
\langle2\rangle_1D(0,1)&0&0                                                                        \\
\langle3\rangle_1D(1,2)&\langle3\rangle_2D(0,2)&0                                                  \\
\langle4\rangle_1D(2,3)&\langle4\rangle_2D(1,3)&\langle4\rangle_3D(0,3)                            \\
\langle5\rangle_1D(3,4)&D\langle5\rangle_2(2,4)&\langle5\rangle_3D(1,4)                            \\
\vdots&\vdots&\vdots\\
\langle n\rangle_1D(n-2,n-1)&\langle n\rangle_2D(n-3,n-1)&\langle n\rangle_3D(n-4,n-1)             \\
\langle n+1\rangle_1D(n-1,n)&\langle n+1\rangle_2D(n-2,n)&\langle n+1\rangle_3D(n-3,n)
\end{matrix}\right.\\
\left.
\begin{matrix}
0&\dotsm& 0&0\\
0&\dotsm& 0&0\\
0&\dotsm& 0&0\\
\langle5\rangle_4D(0,4)&\dotsm& 0&0\\
\vdots&\ddots &\vdots&\vdots\\
\langle n\rangle_4D(n-5,n-1)&\dotsm& \langle n\rangle_{n-1}D(0,n-1)&0\\
\langle n+1\rangle_4D(n-4,n)&\dotsm& \langle n+1\rangle_{n-1}D(1,n)&\langle n+1\rangle_nD(0,n)
\end{matrix}\right).
\end{gather*}
\end{rem}

\begin{rem}
The recursive relation~\eqref{Delan-central-delan} should be new. Therefore, the recursive relation~\eqref{Delannoy-recursive-eq} should be new. Furthermore, the determinantal expression~\eqref{Delannoy-det-exp-eq} should also be new. Interesting!
\end{rem}

\begin{rem}
This paper is a revised version of the electronic preprint~\cite{Delanoy-No.tex}.
\end{rem}

\subsection*{Acknowledgements}
The author thanks Christer Oscar Kiselman (Department of Information Technology, Uppsala University, Sweden) for his careful corrections, helpful suggestions, and valuable comments in April 2020 on the original version of this paper.
\par
The author thanks anonymous referees for their careful corrections and valuable comments on the original version of this paper.

\bibliography{Delanoy-No}

\begin{thebibliography}{10}
\expandafter\ifx\csname url\endcsname\relax
  \def\url#1{\texttt{#1}}\fi
\expandafter\ifx\csname urlprefix\endcsname\relax\def\urlprefix{URL }\fi

\bibitem{MR3908608}
\textsc{Allouche, J.-P.}: \emph{Propri\'{e}t\'{e} de {L}ucas, nombres de
  {D}elannoy et s\'{e}ries formelles alg\'{e}briques}, in: \emph{Les travaux
  combinatoires en {F}rance (1870--1914) et leur actualit\'{e}}, Savoirs Sci.
  Prat. Enseign., Presses Univ. Limoges, Limoges, 2017 pp. 239--254.

\bibitem{Cyril-Sylviane-2005-why}
\textsc{Banderier, C. and Schwer, S.}: \emph{Why {D}elannoy numbers?}, J.
  Statist. Plann. Inference, \textbf{135} (2005), No.~1, 40--54.
\newline\urlprefix\url{https://doi.org/10.1016/j.jspi.2005.02.004}

\bibitem{Booth-Nguyen-2008-09}
\textsc{Booth, R. and Nguyen, H.~D.}: \emph{Bernoulli polynomials and
  {P}ascal's square}, Fibonacci Quart., \textbf{46/47} (2008/09), No.~1,
  38--47.

\bibitem{Bourbaki-Spain-2004}
\textsc{Bourbaki, N.}: \emph{Elements of Mathematics, Functions of a Real
  Variable, Elementary Theory}, Translated from the 1976 French original by
  Philip Spain.

\bibitem{CollegeMJ-2002-Cahill}
\textsc{Cahill, N.~D. and Narayan, D.~A.}: \emph{Fibonacci and {L}ucas numbers
  as tridiagonal matrix determinants}, Fibonacci Quart., \textbf{42} (2004),
  No.~3, 216--221.

\bibitem{MR2410472}
\textsc{Caughman, J.~S., Haithcock, C.~R., and Veerman, J. J.~P.}: \emph{A note
  on lattice chains and {D}elannoy numbers}, Discrete Math., \textbf{308}
  (2008), No.~12, 2623--2628.
\newline\urlprefix\url{https://doi.org/10.1016/j.disc.2007.05.017}

\bibitem{MR2919231}
\textsc{Clapperton, J.~A., Larcombe, P.~J., and Fennessey, E.~J.}: \emph{The
  {D}elannoy numbers: three new non-linear identities}, Bull. Inst. Combin.
  Appl., \textbf{64} (2012), 39--56.

\bibitem{Costabile-Rend-2006}
\textsc{Costabile, F., Dell'Accio, F., and Gualtieri, M.~I.}: \emph{A new
  approach to {B}ernoulli polynomials}, Rend. Mat. Appl. (7), \textbf{26}
  (2006), No.~1, 1--12.

\bibitem{MR3106100}
\textsc{Dziemia\'{n}czuk, M.}: \emph{Generalizing {D}elannoy numbers via
  counting weighted lattice paths}, Integers, \textbf{13} (2013), Paper No.
  A54, 1--33.

\bibitem{MR3734640}
\textsc{Edwards, S. and Griffiths, W.}: \emph{Generalizations of {D}elannoy and
  cross polytope numbers}, Fibonacci Quart., \textbf{55} (2017), No.~4,
  357--366.

\bibitem{Edwards-Griffiths-2020}
\textsc{Edwards, S. and Griffiths, W.}: \emph{On generalized {D}elannoy
  numbers}, J. Integer Seq., \textbf{23} (2020), No.~3, Article 20.3.06,
  13~pages.

\bibitem{MR3908599}
\textsc{Goldstein, C.}: \emph{Henri {A}uguste {D}elannoy et la th\'{e}orie des
  nombres autour de 1900}, in: \emph{Les travaux combinatoires en {F}rance
  (1870--1914) et leur actualit\'{e}}, Savoirs Sci. Prat. Enseign., Presses
  Univ. Limoges, Limoges, 2017 pp. 17--42.

\bibitem{MR2978852}
\textsc{Guo, V. J.~W. and Zeng, J.}: \emph{New congruences for sums involving
  {A}p\'{e}ry numbers or central {D}elannoy numbers}, Int. J. Number Theory,
  \textbf{8} (2012), No.~8, 2003--2016.
\newline\urlprefix\url{https://doi.org/10.1142/S1793042112501138}

\bibitem{MR2293650}
\textsc{Hetyei, G.}: \emph{Central {D}elannoy numbers and balanced
  {C}ohen-{M}acaulay complexes}, Ann. Comb., \textbf{10} (2006), No.~4,
  443--462.
\newline\urlprefix\url{https://doi.org/10.1007/s00026-006-0299-1}

\bibitem{MR2721475}
\textsc{Hetyei, G.}: \emph{Delannoy numbers and {L}egendre polytopes}, in:
  \emph{20th {A}nnual {I}nternational {C}onference on {F}ormal {P}ower {S}eries
  and {A}lgebraic {C}ombinatorics ({FPSAC} 2008)}, Discrete Math. Theor.
  Comput. Sci. Proc., AJ, Assoc. Discrete Math. Theor. Comput. Sci., Nancy,
  2008 pp. 447--458.

\bibitem{Hu-Kim-arXiv-2015}
\textsc{Hu, S. and Kim, M.-S.}: \emph{Two closed forms for the
  {A}postol-{B}ernoulli polynomials}, Ramanujan J., \textbf{46} (2018), No.~1,
  103--117.
\newline\urlprefix\url{https://doi.org/10.1007/s11139-017-9907-4}

\bibitem{ipek}
\textsc{\.{I}pek, A. and Ar\i, K.}: \emph{On {H}essenberg and pentadiagonal
  determinants related with {F}ibonacci and {F}ibonacci-like numbers}, Appl.
  Math. Comput., \textbf{229} (2014), 433--439.
\newline\urlprefix\url{https://doi.org/10.1016/j.amc.2013.12.071}

\bibitem{Janjic-JIS-2012}
\textsc{Janji\'{c}, M.}: \emph{Determinants and recurrence sequences}, J.
  Integer Seq., \textbf{15} (2012), No.~3, Article 12.3.5, 1--21.

\bibitem{MR3168687}
\textsc{Janji\'{c}, M. and Petkovi\'{c}, B.}: \emph{A counting function
  generalizing binomial coefficients and some other classes of integers}, J.
  Integer Seq., \textbf{17} (2014), No.~3, Article 14.3.5, 23~pages.

\bibitem{Kiselman-15-2}
\textsc{Kiselman, C.~O.}: \emph{Estimates for solutions to discrete convolution
  equations}, Mathematika, \textbf{61} (2015), No.~2, 295--308.
\newline\urlprefix\url{https://doi.org/10.1112/S0025579315000108}

\bibitem{Kiselman-17-2}
\textsc{Kiselman, C.~O.}: \emph{Domains of holomorphy for {F}ourier transforms
  of solutions to discrete convolution equations}, Sci. China Math.,
  \textbf{60} (2017), No.~6, 1005--1018.
\newline\urlprefix\url{https://doi.org/10.1007/s11425-015-9029-0}

\bibitem{Kittappa-LAA-1993}
\textsc{Kittappa, R.~K.}: \emph{A representation of the solution of the {$n$}th
  order linear difference equation with variable coefficients}, Linear Algebra
  Appl., \textbf{193} (1993), 211--222.
\newline\urlprefix\url{https://doi.org/10.1016/0024-3795(93)90278-V}

\bibitem{MR3515810}
\textsc{Liu, J.-C.}: \emph{A supercongruence involving {D}elannoy numbers and
  {S}chr\"{o}der numbers}, J. Number Theory, \textbf{168} (2016), 117--127.
\newline\urlprefix\url{https://doi.org/10.1016/j.jnt.2016.04.019}

\bibitem{MR3831408}
\textsc{Liu, J.-C., Li, L., and Wang, S.-D.}: \emph{Some congruences on
  {D}elannoy numbers and {S}chr\"{o}der numbers}, Int. J. Number Theory,
  \textbf{14} (2018), No.~7, 2035--2041.
\newline\urlprefix\url{https://doi.org/10.1142/S1793042118501221}

\bibitem{MR3749744}
\textsc{Mao, G.-S.}: \emph{On sums of binomial coefficients involving {C}atalan
  and {D}elannoy numbers modulo {$p^2$}}, Ramanujan J., \textbf{45} (2018),
  No.~2, 319--330.
\newline\urlprefix\url{https://doi.org/10.1007/s11139-016-9853-6}

\bibitem{MR2887889}
\textsc{Noble, R.}: \emph{Asymptotics of the weighted {D}elannoy numbers}, Int.
  J. Number Theory, \textbf{8} (2012), No.~1, 175--188.
\newline\urlprefix\url{https://doi.org/10.1142/S1793042112500108}

\bibitem{MR2901372}
\textsc{Petersen, K.}: \emph{An adic dynamical system related to the {D}elannoy
  numbers}, Ergodic Theory Dynam. Systems, \textbf{32} (2012), No.~2, 809--823.
\newline\urlprefix\url{https://doi.org/10.1017/S0143385711000472}

\bibitem{Tan-Der-App-Thanks.tex}
\textsc{Qi, F.}: \emph{Derivatives of tangent function and tangent numbers},
  Appl. Math. Comput., \textbf{268} (2015), 844--858.
\newline\urlprefix\url{https://doi.org/10.1016/j.amc.2015.06.123}

\bibitem{Fubini-Polyn.tex-v2}
\textsc{Qi, F.}: \emph{Determinantal expressions and recurrence relations for
  {F}ubini and {E}ulerian polynomials}, J. Interdiscip. Math., \textbf{22}
  (2019), 317--335.
\newline\urlprefix\url{https://doi.org/10.1080/09720502.2019.1624063}

\bibitem{CAM-D-18-00067.tex}
\textsc{Qi, F.}: \emph{A double inequality for the ratio of two non-zero
  neighbouring {B}ernoulli numbers}, J. Comput. Appl. Math., \textbf{351}
  (2019), 1--5.
\newline\urlprefix\url{https://doi.org/10.1016/j.cam.2018.10.049}

\bibitem{andelic-Fonseca-rev.tex}
\textsc{Qi, F.}: \emph{Denying a short proof of a determinantal formula for
  generalized {F}ibonacci polynomials}, J. Math. Anal., \textbf{11} (2020),
  No.~1, 52--57.

\bibitem{Delanoy-No.tex}
\textsc{Qi, F.}: \emph{A determinantal expression and a recursive relation of
  the {D}elannoy numbers}, arXiv preprint,  (2020), 1--9.
\newline\urlprefix\url{https://arxiv.org/abs/2003.12572v1}

\bibitem{2Closed-Bern-Polyn2.tex}
\textsc{Qi, F. and Chapman, R.~J.}: \emph{Two closed forms for the {B}ernoulli
  polynomials}, J. Number Theory, \textbf{159} (2016), 89--100.
\newline\urlprefix\url{https://doi.org/10.1016/j.jnt.2015.07.021}

\bibitem{mathematics-131192.tex}
\textsc{Qi, F. and Guo, B.-N.}: \emph{Some determinantal expressions and
  recurrence relations of the {B}ernoulli polynomials}, Mathematics, \textbf{4}
  (2016), No.~4, Art.~65, 1--11.
\newline\urlprefix\url{https://doi.org/10.3390/math4040065}

\bibitem{Euler-Polyn-Det-Note.tex}
\textsc{Qi, F. and Guo, B.-N.}: \emph{A determinantal expression and a
  recurrence relation for the {E}uler polynomials}, Adv. Appl. Math. Sci.,
  \textbf{16} (2017), No.~9, 297--309.

\bibitem{KJM-706-Turkey.tex}
\textsc{Qi, F. and Guo, B.-N.}: \emph{Explicit and recursive formulas, integral
  representations, and properties of the large {S}chr\"{o}der numbers},
  Kragujevac J. Math., \textbf{41} (2017), No.~1, 121--141.
\newline\urlprefix\url{https://doi.org/10.5937/KgJMath1701121F}

\bibitem{Special-Bell2Euler.tex}
\textsc{Qi, F. and Guo, B.-N.}: \emph{Explicit formulas for special values of
  the {B}ell polynomials of the second kind and for the {E}uler numbers and
  polynomials}, Mediterr. J. Math., \textbf{14} (2017), No.~3, Paper No. 140,
  1--14.
\newline\urlprefix\url{https://doi.org/10.1007/s00009-017-0939-1}

\bibitem{lematema-1389.tex}
\textsc{Qi, F. and Guo, B.-N.}: \emph{Expressing the generalized {F}ibonacci
  polynomials in terms of a tridiagonal determinant}, Matematiche (Catania),
  \textbf{72} (2017), No.~1, 167--175.
\newline\urlprefix\url{https://doi.org/10.4418/2017.72.1.13}

\bibitem{Schroder-Altern.tex}
\textsc{Qi, F. and Guo, B.-N.}: \emph{Some explicit and recursive formulas of
  the large and little {S}chr\"{o}der numbers}, Arab J. Math. Sci., \textbf{23}
  (2017), No.~2, 141--147.
\newline\urlprefix\url{https://doi.org/10.1016/j.ajmsc.2016.06.002}

\bibitem{Nice-ApBern-Polyn-2v.tex}
\textsc{Qi, F. and Guo, B.-N.}: \emph{Two nice determinantal expressions and a
  recurrence relation for the {A}postol-{B}ernoulli polynomials}, J. Indones.
  Math. Soc., \textbf{23} (2017), No.~1, 81--87.
\newline\urlprefix\url{https://doi.org/10.22342/jims.23.1.274.81-87}

\bibitem{AADM-2821.tex}
\textsc{Qi, F. and Guo, B.-N.}: \emph{A diagonal recurrence relation for the
  {S}tirling numbers of the first kind}, Appl. Anal. Discrete Math.,
  \textbf{12} (2018), No.~1, 153--165.
\newline\urlprefix\url{https://doi.org/10.2298/AADM170405004Q}

\bibitem{Motzkin-No-Pass.tex}
\textsc{Qi, F. and Guo, B.-N.}: \emph{Several explicit and recursive formulas
  for generalized {M}otzkin numbers}, AIMS Math., \textbf{5} (2020), No.~2,
  1333--1345.
\newline\urlprefix\url{https://doi.org/10.3934/math.2020091}

\bibitem{clos-hora-poly-sym.tex}
\textsc{Qi, F., K{\i}z{\i}late\c{s}, C., and Du, W.-S.}: \emph{A closed formula
  for the {H}oradam polynomials in terms of a tridiagonal determinant},
  Symmetry, \textbf{11} (2019), No.~6, 1--8.
\newline\urlprefix\url{https://doi.org/10.3390/sym11060782}

\bibitem{Catalan-GF-Plus.tex}
\textsc{Qi, F., Mahmoud, M., Shi, X.-T., and Liu, F.-F.}: \emph{Some properties
  of the {C}atalan--{Q}i function related to the {C}atalan numbers},
  SpringerPlus, \textbf{5} (2016), No. 1126, 1--20.
\newline\urlprefix\url{https://doi.org/10.1186/s40064-016-2793-1}

\bibitem{QLG-RACSAM-Ext.tex}
\textsc{Qi, F., Niu, D.-W., and Guo, B.-N.}: \emph{Some identities for a
  sequence of unnamed polynomials connected with the {B}ell polynomials}, Rev.
  R. Acad. Cienc. Exactas F\'{\i}s. Nat. Ser. A Mat. RACSAM, \textbf{113}
  (2019), No.~2, 557--567.
\newline\urlprefix\url{https://doi.org/10.1007/s13398-018-0494-z}

\bibitem{Schroder-Seq-3rd.tex}
\textsc{Qi, F., Shi, X.-T., and Guo, B.-N.}: \emph{Some properties of the
  {S}chr\"{o}der numbers}, Indian J. Pure Appl. Math., \textbf{47} (2016),
  No.~4, 717--732.
\newline\urlprefix\url{https://doi.org/10.1007/s13226-016-0211-6}

\bibitem{Integers-2196.tex}
\textsc{Qi, F., Shi, X.-T., and Guo, B.-N.}: \emph{Two explicit formulas of the
  {S}chr\"{o}der numbers}, Integers, \textbf{16} (2016), Paper No. A23, 1--15.

\bibitem{IJPA-D-16-00849.tex}
\textsc{Qi, F., Shi, X.-T., and Guo, B.-N.}: \emph{Integral representations of
  the large and little {S}chr\"{o}der numbers}, Indian J. Pure Appl. Math.,
  \textbf{49} (2018), No.~1, 23--38.
\newline\urlprefix\url{https://doi.org/10.1007/s13226-018-0258-7}

\bibitem{UPB-7177.tex}
\textsc{Qi, F., \v{C}er\v{n}anov\'{a}, V., and Semenov, Y.~S.}: \emph{Some
  tridiagonal determinants related to central {D}elannoy numbers, the
  {C}hebyshev polynomials, and the {F}ibonacci polynomials}, Politehn. Univ.
  Bucharest Sci. Bull. Ser. A Appl. Math. Phys., \textbf{81} (2019), No.~1,
  123--136.

\bibitem{Delannoy-Cent-P.tex}
\textsc{Qi, F., \v{C}er\v{n}anov\'{a}, V., Shi, X.-T., and Guo, B.-N.}:
  \emph{Some properties of central {D}elannoy numbers}, J. Comput. Appl. Math.,
  \textbf{328} (2018), 101--115.
\newline\urlprefix\url{https://doi.org/10.1016/j.cam.2017.07.013}

\bibitem{KJM-762-Qi-Wang-Guo.tex}
\textsc{Qi, F., Wang, J.-L., and Guo, B.-N.}: \emph{A representation for
  derangement numbers in terms of a tridiagonal determinant}, Kragujevac J.
  Math., \textbf{42} (2018), No.~1, 7--14.

\bibitem{BIMS-D-18-00265.tex}
\textsc{Qi, F., Wang, J.-L., and Guo, B.-N.}: \emph{A determinantal expression
  for the {F}ibonacci polynomials in terms of a tridiagonal determinant}, Bull.
  Iranian Math. Soc., \textbf{45} (2019), No.~6, 1821--1829.
\newline\urlprefix\url{https://doi.org/10.1007/s41980-019-00232-4}

\bibitem{2rdBern-Det-Zhao.tex}
\textsc{Qi, F. and Zhao, J.-L.}: \emph{Some properties of the {B}ernoulli
  numbers of the second kind and their generating function}, Bull. Korean Math.
  Soc., \textbf{55} (2018), No.~6, 1909--1920.
\newline\urlprefix\url{https://doi.org/10.4134/BKMS.b180039}

\bibitem{Derange-Hess-Det-S.tex}
\textsc{Qi, F., Zhao, J.-L., and Guo, B.-N.}: \emph{Closed forms for
  derangement numbers in terms of the {H}essenberg determinants}, Rev. R. Acad.
  Cienc. Exactas F\'{\i}s. Nat. Ser. A Mat. RACSAM, \textbf{112} (2018), No.~4,
  933--944.
\newline\urlprefix\url{https://doi.org/10.1007/s13398-017-0401-z}

\bibitem{MR2893830}
\textsc{Razpet, M.}: \emph{Sequence alignment and {D}elannoy numbers}, Obzornik
  Mat. Fiz., \textbf{58} (2011), No.~4, 133--145.

\bibitem{MR2598470}
\textsc{Samieinia, S.}: \emph{The number of continuous curves in digital
  geometry}, Port. Math., \textbf{67} (2010), No.~1, 75--89.
\newline\urlprefix\url{https://doi.org/10.4171/PM/1858}

\bibitem{MR2374122}
\textsc{Schr\"{o}der, J.}: \emph{Delannoy and tetrahedral numbers}, Comment.
  Math. Univ. Carolin., \textbf{48} (2007), No.~3, 389--394.

\bibitem{MR1971435}
\textsc{Sulanke, R.~A.}: \emph{Objects counted by the central {D}elannoy
  numbers}, J. Integer Seq., \textbf{6} (2003), No.~1, Article 03.1.5, 1--19.

\bibitem{MR2832831}
\textsc{Sun, Z.-W.}: \emph{On {D}elannoy numbers and {S}chr\"{o}der numbers},
  J. Number Theory, \textbf{131} (2011), No.~12, 2387--2397.
\newline\urlprefix\url{https://doi.org/10.1016/j.jnt.2011.06.005}

\bibitem{MR3715232}
\textsc{Sun, Z.-W.}: \emph{Arithmetic properties of {D}elannoy numbers and
  {S}chr\"{o}der numbers}, J. Number Theory, \textbf{183} (2018), 146--171.
\newline\urlprefix\url{https://doi.org/10.1016/j.jnt.2017.07.011}

\bibitem{MR2460233}
\textsc{T\u{a}rn\u{a}uceanu, M.}: \emph{The number of fuzzy subgroups of finite
  cyclic groups and {D}elannoy numbers}, European J. Combin., \textbf{30}
  (2009), No.~1, 283--287.
\newline\urlprefix\url{https://doi.org/10.1016/j.ejc.2007.12.005}

\bibitem{Turnbull-book-1960}
\textsc{Turnbull, H.~W.}: \emph{The Theory of Determinants, Matrices, and
  Invariants}, 3rd ed, Dover Publications, Inc., New York, 1960.

\bibitem{MR1291893}
\textsc{Vassilev, M. and Atanassov, K.}: \emph{On {D}elanoy numbers}, Annuaire
  Univ. Sofia Fac. Math. Inform., \textbf{81} (1987), No.~1, 153--162 (1994).

\bibitem{MR3952588}
\textsc{Wang, Y., Zheng, S.-N., and Chen, X.}: \emph{Analytic aspects of
  {D}elannoy numbers}, Discrete Math., \textbf{342} (2019), No.~8, 2270--2277.
\newline\urlprefix\url{https://doi.org/10.1016/j.disc.2019.04.003}

\bibitem{Euler-No-3Sum.tex}
\textsc{Wei, C.-F. and Qi, F.}: \emph{Several closed expressions for the
  {E}uler numbers}, J. Inequal. Appl.,  (2015), No. Paper No.~219, 1--8.
\newline\urlprefix\url{https://doi.org/10.1186/s13660-015-0738-9}

\end{thebibliography}
\bibliographystyle{mmn}

\end{document}